\documentclass[12pt]{article}

\usepackage[margin=1in]{geometry}
\usepackage{graphicx}
\usepackage{amsmath,amssymb,amsfonts}
\usepackage{amsthm}
\usepackage[title]{appendix}

\theoremstyle{plain}
\newtheorem{theorem}{Theorem}
\newtheorem{proposition}[theorem]{Proposition}
\newtheorem{lemma}[theorem]{Lemma}
\newtheorem{corollary}[theorem]{Corollary}

\theoremstyle{definition}
\newtheorem{definition}[theorem]{Definition}
\newtheorem{remark}[theorem]{Remark}

\begin{document}

\title{Voronoi limit measures for iterates of constant-coefficient differential operators on rational 
functions with simple poles}

\author{Bosco Nyandwi\thanks{boscobos715@gmail.com, Department of Mathematics, University of Rwanda, KN 67 Nyarugenge, Kigali 3900, Rwanda} \and 
Christian H\"agg\thanks{hagg@math.su.se, Department of Mathematics, Stockholm University, SE-106 91 Stockholm, Sweden} \and 
Celestin Kurujyibwami\thanks{celeku@yahoo.fr, Department of Mathematics, University of Rwanda} \and 
Leon Fidele Ruganzu Uwimbabazi\thanks{ruganzu01@gmail.com, Department of Mathematics, University of Rwanda}}

\date{}

\maketitle

\begin{abstract}
B\o gvad and H\"agg proved that for a rational function with simple poles, the zeros of successive 
derivatives accumulate on the Voronoi diagram of the pole set, and the normalized zero-counting measures 
converge to a canonical probability measure supported on this diagram. We extend this result from pure 
derivatives to iterates of an arbitrary monic constant-coefficient differential operator.

Let $h(z)=A(z)/B(z)$ be a reduced rational function, where $B$ is monic of degree $b\ge2$ with distinct
 zeros $S=\{z_1,\dots,z_b\}$, and let $P(D)=\sum_{j=0}^m c_jD^j$ be a monic constant-coefficient 
differential operator of order $m\ge1$. After clearing denominators, we can write
 $P(D)^n(h)=\widetilde A_n/B^{mn+1}$ and study the zeros of the numerator polynomials $\widetilde A_n$. 
If $r:=\min\{j:c_j\neq0\}$, then (after passing to the proper part of $h$ when $r>0$) the associated 
zero-counting measures converge vaguely to
$$\frac{m(b-1)}{bm-r}\,\mu_S,$$
where $\mu_S$ is the B\o gvad--H\"agg probability measure supported on the Voronoi diagram $V_S$.
 In particular, the limit is a probability measure exactly when $P(D)=D^m$; otherwise a proportion
 $\frac{m-r}{bm-r}$ of zeros escapes to infinity (in the sense of vague convergence). When $r<m$,
 the unshifted logarithmic potentials diverge, but an explicit factorial renormalization yields 
$L^1_{\mathrm{loc}}(\mathbb C)$ convergence to a subharmonic limit with Riesz measure 
$\frac{m(b-1)}{bm-r}\,\mu_S$. Apart from this scalar factor, the limiting measure is determined 
solely by the pole configuration; the coefficients of $P(D)$ affect only an additive constant in the
 limiting potential.
\end{abstract}


\noindent\textbf{MSC Classification 2020:} 30C15, 31A15

\section{Introduction}\label{sec1}

In 1922, George P\'olya introduced the \emph{final set} of a meromorphic function and proved that all finite limit points of zeros of successive derivatives lie in this set \cite{GP}; see also \cite{Polya}.
For rational functions with simple poles, this final set is closely related to the Voronoi diagram of the pole set \cite{Hayman}.
B\o gvad and H\"agg showed that for such rational functions the normalized zero-counting measures of the numerator polynomials of $f^{(n)}$ converge to a canonical probability measure supported on the Voronoi diagram \cite{B:H}.
H\"agg later extended this to meromorphic functions of the form $f=(A/B)e^T$, with a factorial renormalization in the logarithmic potentials \cite{Hagg}.
See \cite{PJS,PJS1,Weiss,BSTW,Robert} for further extensions and related work.

In the present paper we study iterates of monic constant-coefficient differential operators acting on rational functions with simple poles.
Let $h(z)=A(z)/B(z)$ be reduced (i.e., $\gcd(A,B)=1$), where $B$ is monic of degree $b\ge2$ with distinct zeros $z_1,\dots,z_b$.
Let
$$P(D)=\sum_{j=0}^{m}c_jD^j,\qquad m\ge1,\qquad c_m=1,$$
be a monic constant-coefficient differential operator, where $D=\frac{d}{dz}$, and write its symbol as
$$q(x)=\sum_{j=0}^{m}c_jx^j.$$
Set $r:=\operatorname{ord}_0 q\in\{0,\dots,m\}$, i.e., the smallest index $j$ with $c_j\neq0$ (so $c_0=\cdots=c_{r-1}=0$ and $c_r\neq0$). Equivalently, $q(x)=x^r\widehat q(x)$ with $\widehat q(0)=c_r\neq0$, so $P(D)=D^r\widehat q(D)$.
The assumption that $B$ and $P(D)$ are monic is only a normalization: one may always scale $A$ and $B$ so that $B$ becomes monic, and multiplying $P(D)$ by a nonzero constant $\alpha$ only scales each iterate $P(D)^n(h)$ by $\alpha^n$.

Clearing denominators, for each $n\ge1$ we may write
$$P(D)^n(h)=\frac{\widetilde A_n(z)}{B(z)^{mn+1}},$$
for a polynomial numerator $\widetilde A_n$; in fact $\gcd(\widetilde A_n,B)=1$ (Lemma~\ref{lem_deg_Pm}).
We study the zeros of $\widetilde A_n$ via the normalized zero-counting measures $\mu_n:=\mu_{\widetilde A_n}$ (Definition~\ref{def:zcmeasure}).

Our main results (Theorems~\ref{thr5:mainm} and \ref{thr:r_positive}) describe the asymptotic zero distribution of $\widetilde A_n$.
Set $S:=\{z_1,\dots,z_b\}$ and let $V_S$ be its Voronoi diagram.

When $r>0$, the iterates $P(D)^n$ eventually annihilate the polynomial part of $h$.
More precisely, writing $h=Q+R/B$ with $Q$ a polynomial and $\deg R<b$ (so $R/B$ is proper), Lemma~\ref{lem:kill_poly_part} shows that $P(D)^n(h)=P(D)^n(R/B)$ for all sufficiently large $n$.
Thus, for asymptotic questions, we may (and will) assume $h$ is proper when $r>0$.

Lemma~\ref{lem_deg_Pm} shows that the numerator degrees satisfy
$$d_n:=\deg(\widetilde A_n)=a+n(bm-r),\qquad a=\deg A$$
(in particular, $d_n=a+bmn$ when $r=0$).
Then $\mu_n$ converges vaguely on $\mathbb C$ to the canonical subprobability measure
\begin{equation}\label{eq:main_limit_measure}
\mu_{c,r}=\frac{m(b-1)}{bm-r}\,\mu_S,
\end{equation}
where $\mu_S$ is the B\o gvad--H\"agg probability measure supported on $V_S$ (Subsection~\ref{sub:sec2_BH}); the subscript $c$ stands for ``canonical'' (not to be confused with the operator coefficients $c_j$).
In particular, $\mu_{c,r}$ is a subprobability measure of total mass $m(b-1)/(bm-r)$, and it is a probability measure exactly in the pure derivative case $P(D)=D^m$ (i.e., $r=m$).
Apart from this dependence on $r$, the limiting measure depends only on the pole configuration; the coefficients of $P(D)$ affect the limit potential only through an additive constant (see Theorems~\ref{thr5:mainm} and \ref{thr:r_positive}).

If $r<m$, the logarithmic potentials $\mathcal L_{\mu_n}$ diverge to $+\infty$ off $V_S$, but after subtracting $\bigl(\log((mn)!)-\log((rn)!)\bigr)/d_n$ one obtains $L^1_{\mathrm{loc}}(\mathbb C)$ convergence to an explicit subharmonic limit whose Riesz measure is $\mu_{c,r}$.

When $b=1$ (a single simple pole), the Voronoi diagram is empty. If $r=m$ (so $P(D)=D^m$) then the numerator is constant, while if $r<m$ all zeros escape to infinity; this degenerate case is recorded in Appendix~\ref{app:onepole}.
For $m=1$ this specializes to H\"agg's theorem \cite{Hagg} (when $r=0$, i.e., $c_0\neq0$) and to the B\o gvad--H\"agg theorem \cite{B:H} (when $P(D)=D$, i.e., $r=m=1$).

The paper is organized as follows.
Section~\ref{sec2} reviews the needed background on Voronoi diagrams, logarithmic potentials, and the B\o gvad--H\"agg measure.
Section~\ref{sec3} recalls the main results of \cite{B:H,Hagg}.
Sections~\ref{sec4}--\ref{sec5} prove our main theorems for general monic constant-coefficient operators $P(D)$: Section~\ref{sec4} treats the case $r=0$ (equivalently, $c_0\neq0$), while Section~\ref{sec5} treats $r\ge1$.
Section~\ref{sec6} concludes with remarks and further directions, while Appendix~\ref{app:onepole} records the degenerate one-pole case.

\section{Preliminaries}\label{sec2}

\subsection{Voronoi diagrams}\label{sub:sec2}

Let $S=\{z_1,\dots,z_b\}\subset\mathbb{C}$ be a finite set of distinct points and define the distance function
\begin{equation}\label{psi_def}
\psi_S(z):=\min_{1\le i\le b}|z-z_i|.
\end{equation}
For each $i$, the (closed) \emph{Voronoi cell} of $z_i$ is
$$V_i:=\{z\in\mathbb{C}:\psi_S(z)=|z-z_i|\}.$$
We also write
$$V_i^\circ:=\{z\in\mathbb{C}:|z-z_i|<|z-z_j|\ \text{for all }j\ne i\}$$
for the corresponding \emph{open} cell (where the nearest site is uniquely $z_i$).
Note that the closed cells cover the plane, $\bigcup_{i=1}^b V_i=\mathbb C$, and that the open cells $V_i^\circ$ are pairwise disjoint.

The \emph{Voronoi diagram} (or \emph{Voronoi skeleton}) associated with $S$ is the closed set
$$V_S:=\{z\in\mathbb{C}:\text{the minimum in \eqref{psi_def} is attained for at least two indices}\}
=\bigcup_{1\le i<j\le b}V_{ij},$$
where each \emph{Voronoi edge} is
$$V_{ij}:=\{z\in\mathbb{C}:|z-z_i|=|z-z_j|=\psi_S(z)\}.$$
Equivalently, $z\notin V_S$ if and only if the nearest site is unique, in which case $z\in V_i^\circ$ for exactly one index $i$.
The \emph{Voronoi vertices} are those points where the minimum is attained by at least three indices.

Figure~\ref{f0} shows an example. The function $\psi_S$ is piecewise $C^1$, with locus of non-differentiability contained in $S\cup V_S$.

\begin{figure}[!hbt]
\centering
\includegraphics[width=0.65\linewidth]{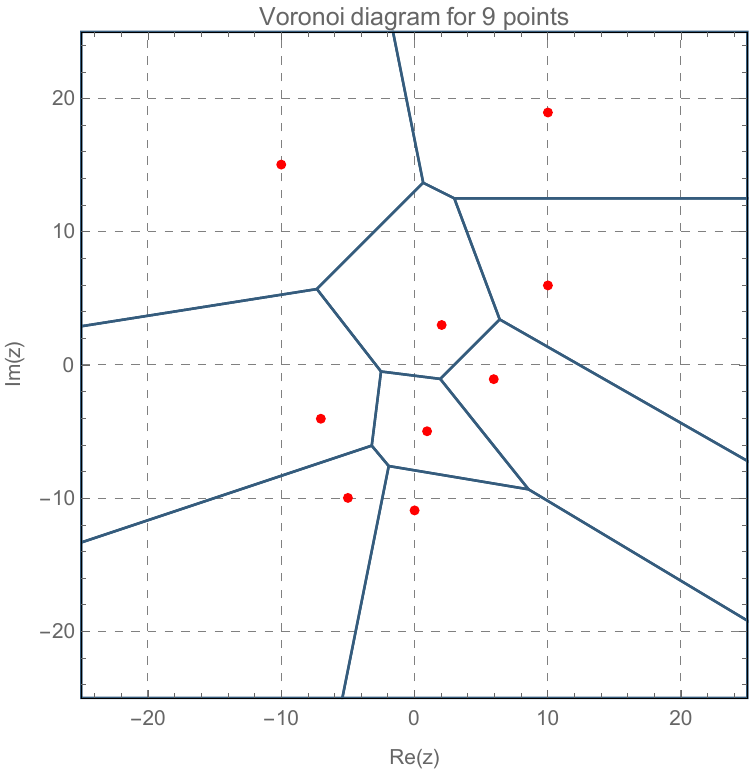}
\caption{The Voronoi diagram for the nine points given by the zeros of
$f(z)=(z+5+10i)(z-10-6i)(z-1+5i)(z-6+i)(z-2-3i)(z+7+4i)(z+11i)(z-10-19i)(z+10-15i)$.
The Voronoi edges are shown (bounded edges as line segments, unbounded edges as rays).}\label{f0}
\end{figure}

Voronoi diagrams appear in many contexts (see, e.g., \cite{A:FR}); here we only use their basic geometric characterization via the distance function $\psi_S$ and the induced Voronoi cells and edges.

To describe asymptotic zero distributions we will work with normalized zero-counting measures of polynomials and vague convergence of measures; we recall these notions next.

\begin{definition}[Zero-counting measure]\label{def:zcmeasure}
Let $Q$ be a nonconstant polynomial of degree $k$. Write
$$Q(z)=\mathrm{LC}(Q)\prod_{j=1}^{s}(z-\zeta_j)^{m_j},$$
where $\zeta_1,\dots,\zeta_s$ are the distinct zeros of $Q$ and $\sum_{j=1}^s m_j=k$. The associated \emph{zero-counting measure} is
\begin{equation}\label{zcmeas1}
\mu_Q:=\frac{1}{k}\sum_{j=1}^{s} m_j\,\delta_{\zeta_j},
\end{equation}
where $\delta_{\zeta}$ denotes the unit point mass at $\zeta$. In particular, $\mu_Q(\mathbb C)=1$.
\end{definition}

\begin{remark}\label{rem:logpot_poly}
The logarithmic potential of $\mu_Q$ can be expressed directly in terms of $Q$:
\begin{equation}\label{logpot_poly}
\mathcal L_{\mu_Q}(z)=\frac{1}{k}\log|Q(z)|-\frac{1}{k}\log|\mathrm{LC}(Q)|,
\end{equation}
where $\mathrm{LC}(Q)$ denotes the leading coefficient of $Q$.
\end{remark}

\begin{definition}[Vague convergence]\label{def:vague}
A sequence of finite Borel measures $\{\nu_n\}$ on $\mathbb C$ is said to converge \emph{vaguely} to a finite Borel measure $\nu$ if
$$\int_{\mathbb C}\varphi\,d\nu_n\longrightarrow \int_{\mathbb C}\varphi\,d\nu
\qquad\text{for every }\varphi\in C_c(\mathbb C).$$
(Here $C_c(\mathbb C)$ denotes the continuous functions with compact support.)
In particular, vague limits of probability measures may have total mass strictly less than $1$ due to escape of mass to infinity; such limits are (finite) \emph{subprobability measures}.
\end{definition}

\subsection{Notation}
Throughout, we use the falling factorial notation
$$(x)_k := x(x-1)\cdots(x-k+1),\qquad k\ge0,$$
with the convention $(x)_0=1$.

\subsection{Logarithmic potentials and distributional Laplacians}\label{sub:sec2_logpot}

For each finite positive Borel measure $\mu$ on $\mathbb{C}$ with compact support (in particular, for the zero-counting measures considered below), its logarithmic potential is defined by
\begin{equation}\label{tpm2}
\mathcal L_{\mu}(z):=\int_{\mathbb{C}} \log|z-w|\,d\mu(w).
\end{equation}
This takes values in $[-\infty,\infty)$ and is finite on $\mathbb C\setminus \mathrm{supp}\,\mu$.
Then $\mathcal L_{\mu}\in L^1_{\mathrm{loc}}(\mathbb C)$ and, in the sense of distributions,
\begin{equation}\label{tpm4}
\mu=\frac{1}{2\pi}\Delta \mathcal L_{\mu}.
\end{equation}

More generally, whenever $u\in L^1_{\mathrm{loc}}(\mathbb C)$ is subharmonic and $\mu=(2\pi)^{-1}\Delta u$, we will also refer to $u$ as a logarithmic potential of $\mu$. In particular, the B\o gvad--H\"agg potential $\Xi$ defined below in Subsection~\ref{sub:sec2_BH} is a logarithmic potential of its Riesz measure $\mu_S$ (even though $\mathrm{supp}\,\mu_S\subseteq V_S$ is unbounded).

\subsection{Harmonic and subharmonic functions}\label{sub:sec2_harm}

We use standard notions from potential theory in the complex plane; see e.g.\ \cite{S.E.SR}.
Let $\Omega\subset\mathbb C$ be open, and write $\mathbb{D}(z_0,\rho)$ for the open disc of radius $\rho$ centered at $z_0$.
A function $u:\Omega\to[-\infty,\infty)$ which is not identically $-\infty$ is called \emph{subharmonic} in $\Omega$
if it is upper semicontinuous and satisfies the sub-mean value property:
for every $z_0\in\Omega$ and every $\rho>0$ with $\overline{\mathbb{D}(z_0,\rho)}\subset \Omega$,
$$u(z_0)\le \frac{1}{2\pi}\int_{0}^{2\pi} u(z_0+\rho e^{it})\,dt.$$
Equivalently, $u\in L^1_{\mathrm{loc}}(\Omega)$ and $\Delta u$ is a positive distribution.
In particular, if $u\in C^2(\Omega)$ then $u$ is subharmonic if and only if $\Delta u\ge0$ in $\Omega$.

We will use that the maximum of finitely many subharmonic functions is subharmonic, and that
$$\Delta\log|z|=2\pi\delta_0$$
in the sense of distributions.

\subsection{The B\o gvad--H\"agg potential and measure}\label{sub:sec2_BH}

Let $S=\{z_1,\ldots,z_b\}\subset\mathbb C$ be the set of distinct zeros of a monic polynomial
$$B(z)=\prod_{i=1}^{b}(z-z_i),\qquad b\ge 2.$$
With $\psi_S$ as in \eqref{psi_def}, define for $z\in\mathbb C\setminus S$
\begin{equation}\label{defXi}
\Xi(z):=\frac{1}{b-1}\bigl(-\log\psi_S(z)+\log|B(z)|\bigr)
=\frac{1}{b-1}\left(\max_{1\le i\le b}\log\frac{1}{|z-z_i|}+\log|B(z)|\right),
\end{equation}
where we used $-\log\min_i|z-z_i|=\max_i\log\frac{1}{|z-z_i|}$.
On the open Voronoi cell $V_i^\circ$ (i.e.\ where $\psi_S(z)=|z-z_i|$ and the minimizer in \eqref{psi_def} is unique) we have
$$\Xi(z)=\frac{1}{b-1}\sum_{j\ne i}\log|z-z_j|,$$
which is harmonic on $V_i^\circ$ and, since it has no singularity at $z_i$, extends harmonically to a neighborhood of $z_i$.
Since adjacent cells give the same boundary values on their common edge, $\Xi$ extends to a continuous subharmonic function on all of $\mathbb C$; we keep the notation $\Xi$ for this extension.

\begin{proposition}\label{prop0}
The function $\Xi$ is continuous and subharmonic on $\mathbb C$, harmonic on the interior of each Voronoi cell, and its Riesz measure
\begin{equation}\label{zcma1}
\mu_S:=\frac{1}{2\pi}\Delta\Xi
\end{equation}
is a probability measure supported on the Voronoi diagram $V_S$.
\end{proposition}

\begin{proof}
See \cite[Prop.~2.2]{B:H}.
\end{proof}

\begin{remark}
Even though the Voronoi diagram $V_S$ is unbounded, the measure $\mu_S$ is finite (in fact a probability measure). Indeed, as $|z|\to\infty$ we have $\psi_S(z)=|z|+O(1)$ and $B(z)=z^b+O(z^{b-1})$, hence
$$\Xi(z)=\frac{1}{b-1}\bigl(-\log\psi_S(z)+\log|B(z)|\bigr)=\log|z|+O(1).$$
In general, if a subharmonic function $u$ satisfies $u(z)=M\log|z|+O(1)$ as $|z|\to\infty$, then its Riesz measure $(2\pi)^{-1}\Delta u$ has total mass $M$. Applying this to $\Xi$ gives $\mu_S(\mathbb C)=1$.
\end{remark}

\section{Recent related results}\label{sec3}

\subsection{Results of B\o gvad and H\"agg}\label{subsec31}

Based on this canonical measure defined by \eqref{zcma1}, we can restate an important result proved in \cite{B:H} as follows:

\begin{theorem}[R. B\o gvad--Ch. H\"agg]\label{thr1}
Let $f=A/B$ be a reduced rational function, where $B$ is a monic polynomial of degree
$b\geq 2$ with distinct zeros $z_1,\ldots,z_b$, and set $S:=\{z_1,\ldots,z_b\}$.
For each $n\ge1$ let $\mu_n$ be the (normalized) zero-counting measure of the numerator polynomial of $f^{(n)}$ (Definition~\ref{def:zcmeasure}).
Then, as $n\to\infty$:
\begin{enumerate}
\item[(a)] the measures $\mu_n$ converge vaguely on $\mathbb C$ to the probability measure $\mu_S$ supported on the Voronoi diagram of $S$;
\item[(b)] the logarithmic potentials $\mathcal L_{\mu_n}$ converge in $L^{1}_{\mathrm{loc}}(\mathbb C)$ to the logarithmic potential of $\mu_S$, namely the B\o gvad--H\"agg potential $\Xi$ from \eqref{defXi}.
\end{enumerate}
\end{theorem}
See \cite{B:H} for numerical illustrations.

H\"agg later generalized Theorem~\ref{thr1} to meromorphic functions of the form $f=(A/B)e^T$, where $A,B,T$ are polynomials with $b=\deg B\ge2$ and $t=\deg T\ge1$ \cite{Hagg}.
His theorem identifies an explicit canonical subprobability measure supported on the Voronoi diagram of the poles of $f$ (i.e., the zeros of $B$) and describes the asymptotic zero distribution of $f^{(n)}$.
We record his result in the form used below.
\begin{theorem}[Ch. H\"agg]\label{thr2}
Let
$$f(z)=\frac{A(z)}{B(z)}e^{T(z)},$$
where $A,B,T$ are polynomials of degrees $a,b,t$, respectively.
Assume $b\ge2$, $t\ge1$, $\gcd(A,B)=1$, and that $B$ is monic with distinct zeros $z_1,\ldots,z_b$; set $S:=\{z_1,\ldots,z_b\}$.
For each $n\ge1$ let $\mu_n$ be the (normalized) zero-counting measure of the numerator polynomial of $f^{(n)}$ (Definition~\ref{def:zcmeasure}).
Then the following holds.
\begin{enumerate}
	\item[(i)] The measures $\mu_n$ converge vaguely on $\mathbb C$ to the canonical subprobability measure
	$$\mu_{\mathrm{can}}=\frac{b-1}{b+t-1}\,\mu_S,$$
	supported on the Voronoi diagram of $S$; in particular $\mu_{\mathrm{can}}(\mathbb C)=\frac{b-1}{b+t-1}$.
	\item[(ii)] The logarithmic potentials $\mathcal L_{\mu_n}(z)$ diverge as $n\to\infty$.
	\item[(iii)] The shifted logarithmic potentials
	$$\tilde{\mathcal L}_{\mu_n}(z):=\mathcal L_{\mu_n}(z)-\frac{\log(n!)}{n(b+t-1)+a}$$
	converge in $L^{1}_{\mathrm{loc}}(\mathbb C)$ to the logarithmic potential of $\mu_{\mathrm{can}}$,
	$$\widehat{\Xi}(z)=\dfrac{1}{b+t-1}\left(\max_{1\le i\le b}\log\frac{1}{|z-z_i|}
		+\log \left|B(z)\right|-\log\left(\left|d_t\right|t\right)\right),$$
	where $T(z)=d_t z^t+\cdots$ with $d_t\neq0$.
\end{enumerate}
In particular, $\mu_{\mathrm{can}}=(2\pi)^{-1}\Delta \widehat{\Xi}$ in the sense of distributions.
\end{theorem}
\begin{proof}
See \cite{Hagg}.
\end{proof}

\section{The case $r=0$ (equivalently, $c_0\neq0$)}\label{sec4}

In this section we treat the case $r=0$ (equivalently $c_0=q(0)\neq0$) and prove Theorem~\ref{thr5:mainm}. We also record a growth estimate for $P(D)^n(h)$ (Lemma~\ref{lem_growth_m}) which does not use $c_0\neq0$ and will be reused in Section~\ref{sec5}.

Let
$$P(D)=D^m+c_{m-1}D^{m-1}+\cdots+c_0,\qquad c_0\neq0,$$
and write $q(x):=x^m+c_{m-1}x^{m-1}+\cdots+c_0$.
As before, let
$$h(z)=\frac{A(z)}{B(z)},\qquad \gcd(A,B)=1,$$
where $\deg A=a$ and where $B$ is monic of degree $b\ge2$ with pairwise distinct zeros $z_1,\dots,z_b$.

\subsection{Derivative numerators for $h=A/B$}

\begin{lemma}\label{lem:Ak}
Let $h(z)=A(z)/B(z)$ be reduced and write $B(z)=\prod_{i=1}^b (z-z_i)$ with distinct zeros.
Define polynomials $A_k$ recursively by $A_0:=A$ and
$$A_{k+1}(z):=A_k'(z)\,B(z)-(k+1)\,A_k(z)\,B'(z),\qquad k\ge0.$$
Then
$$h^{(k)}(z)=\frac{A_k(z)}{B(z)^{k+1}},\qquad k\ge0,$$
and for every $i\in\{1,\dots,b\}$ we have
$$A_k(z_i)=(-1)^k\,k!\,A(z_i)\,B'(z_i)^k\neq 0,\qquad k\ge0.$$
\end{lemma}

\begin{proof}
The identity for $h^{(k)}$ follows by induction using the quotient rule:
if $h^{(k)}=A_k/B^{k+1}$ then
$$h^{(k+1)}=\left(\frac{A_k}{B^{k+1}}\right)'
=\frac{A_k'B-(k+1)A_kB'}{B^{k+2}}
=\frac{A_{k+1}}{B^{k+2}}.$$
If $B(z_i)=0$ and $B'(z_i)\neq0$, then $A(z_i)\neq0$ since $\gcd(A,B)=1$.
Evaluating the recursion at $z_i$ gives
$A_{k+1}(z_i)=-(k+1)A_k(z_i)B'(z_i)$, and hence
$A_k(z_i)=(-1)^k k! A(z_i)B'(z_i)^k$.
\end{proof}

\subsection{Reduction, degree and leading coefficient}

\begin{lemma}\label{lem:Ak_degree_LC}
Let $h=A/B$ be as in Lemma~\ref{lem:Ak}, and assume additionally that $\deg A=a<b=\deg B$ (i.e., $h$ is \emph{proper}, meaning it vanishes at infinity) and that $B$ is monic.
Let $A_k$ be defined as in Lemma~\ref{lem:Ak}.
Then for every $k\ge0$,
\begin{equation}\label{eq:Ak_deg}
\deg(A_k)=a+k(b-1),
\end{equation}
and the leading coefficients satisfy
\begin{equation}\label{eq:Ak_LC}
\mathrm{LC}(A_k)=\mathrm{LC}(A)\,(a-b)_k.
\end{equation}
\end{lemma}

\begin{proof}
We argue by induction.
For $k=0$ the claims are trivial.
Assume \eqref{eq:Ak_deg} and \eqref{eq:Ak_LC} hold for some $k$ and set $d_k:=\deg(A_k)=a+k(b-1)$ and $L_k:=\mathrm{LC}(A_k)$.
Since $B$ is monic, the leading term of $A_k'B$ is $d_kL_k\,z^{d_k+b-1}$, while the leading term of $(k+1)A_kB'$ is $(k+1)bL_k\,z^{d_k+b-1}$.
Hence the leading coefficient of
$$A_{k+1}=A_k'B-(k+1)A_kB'$$
equals $(d_k-(k+1)b)L_k=(a-b-k)L_k$.
Because $a<b$, we have $a-b-k\neq0$ for all $k\ge0$, so there is no cancellation at the top degree and therefore
$\deg(A_{k+1})=d_k+b-1=a+(k+1)(b-1)$.
Moreover,
$L_{k+1}=(a-b-k)L_k$, which yields \eqref{eq:Ak_LC}.
\end{proof}

\begin{lemma}\label{lem_deg_Pm}
Let $h=A/B$ be reduced with $\deg A=a$, and let $B$ be monic of degree $b\ge1$ with simple zeros $z_1,\dots,z_b$.
Let $P(D)=\sum_{j=0}^{m}c_jD^j$ be a monic constant-coefficient differential operator with symbol $q(x)=\sum_{j=0}^{m}c_jx^j$.
Set $r:=\operatorname{ord}_0 q$ and let $c_r$ be the first nonzero coefficient of $q$.
Assume either $r=0$, or $r>0$ and $h$ is proper (i.e.\ $\deg A=a<b$).
(When $r>0$, this causes no loss of generality for the large-$n$ asymptotics, since the polynomial part of $h$ is annihilated by $P(D)^n$ for all sufficiently large $n$; see Lemma~\ref{lem:kill_poly_part}.)
Write
$$q(x)^n=\sum_{k=0}^{mn} p_{n,k}\,x^k,$$
let $A_k$ be as in Lemma~\ref{lem:Ak}, and define
$$\widetilde A_n(z):=\sum_{k=0}^{mn} p_{n,k}\,A_k(z)\,B(z)^{mn-k}.$$
Then
$$P(D)^n(h(z))=\frac{\widetilde A_n(z)}{B(z)^{mn+1}},\qquad \gcd(\widetilde A_n,B)=1.$$
Moreover,
$$\deg(\widetilde A_n)=a+n(bm-r)$$
and
$$\mathrm{LC}(\widetilde A_n)=c_r^{\,n}\,\mathrm{LC}(A)\,(a-b)_{rn}.$$
\end{lemma}

\begin{proof}
Write $P(D)^n=q(D)^n=\sum_{k=0}^{mn}p_{n,k}D^k$.
Lemma~\ref{lem:Ak} gives $D^k h=A_k/B^{k+1}$, and bringing to the denominator $B^{mn+1}$ yields the stated representation.

If $z_i$ is a zero of $B$, then $B(z_i)=0$ forces
$$\widetilde A_n(z_i)=p_{n,mn}A_{mn}(z_i).$$
Since $q$ is monic we have $p_{n,mn}=1$, and Lemma~\ref{lem:Ak} gives $A_{mn}(z_i)\neq0$.
Hence $\widetilde A_n(z_i)\neq0$, so $\gcd(\widetilde A_n,B)=1$.

Set $k_0:=rn$ (so $k_0=0$ when $r=0$). Then $p_{n,k}=0$ for $k<k_0$ and $p_{n,k_0}=c_r^{\,n}\neq0$.
A direct induction from $A_{k+1}=A_k'B-(k+1)A_kB'$ gives $\deg(A_k)\le a+k(b-1)$, and thus
$$\deg\bigl(A_kB^{mn-k}\bigr)\le a+k(b-1)+b(mn-k)=a+bmn-k.$$
Therefore the maximal degree in the sum defining $\widetilde A_n$ is attained at the smallest $k$ with $p_{n,k}\neq0$, namely $k=k_0$.
If $r=0$ this gives $\deg(\widetilde A_n)=a+bmn$.
If $r>0$, then $h$ is proper and Lemma~\ref{lem:Ak_degree_LC} gives $\deg(A_{k_0})=a+k_0(b-1)$, so the $k=k_0$ term has degree $a+bmn-k_0=a+n(bm-r)$ and all terms with $k>k_0$ have strictly smaller degree.
This proves $\deg(\widetilde A_n)=a+n(bm-r)$.

The same degree comparison shows that the top-degree coefficient comes only from the term with $k=k_0$.
Since $B$ is monic, this yields $\mathrm{LC}(\widetilde A_n)=p_{n,0}\mathrm{LC}(A)=c_0^{\,n}\mathrm{LC}(A)$ when $r=0$ (using $(a-b)_0=1$), while for $r>0$ we have $\mathrm{LC}(A_{k_0})=\mathrm{LC}(A)\,(a-b)_{k_0}$ (Lemma~\ref{lem:Ak_degree_LC}), hence
$\mathrm{LC}(\widetilde A_n)=p_{n,k_0}\mathrm{LC}(A_{k_0})=c_r^{\,n}\mathrm{LC}(A)(a-b)_{rn}$.
\end{proof}

\subsection{A growth estimate for $P(D)^n(h)$}

Let $S:=\{z_1,\dots,z_b\}$ denote the pole set, and recall that $\psi_S(z)=\min_{1\le i\le b}|z-z_i|$ (see \eqref{psi_def}).

\begin{lemma}\label{lem_growth_m}
Let $h=A/B$ be as above and let $P(D)=D^m+c_{m-1}D^{m-1}+\cdots+c_0$.
Then for any $z\in\mathbb C\setminus S$,
$$\limsup_{n\to\infty}\left|\frac{P(D)^n\!\left(h(z)\right)}{(mn)!}\right|^{1/n}\le \psi_S(z)^{-m}.$$
Moreover, if the nearest pole to $z$ is unique, say $\psi_S(z)=|z-z_{i_0}|$, then with $\alpha_{i_0}:=\operatorname{Res}(h,z_{i_0})=A(z_{i_0})/B'(z_{i_0})$ we have the sharper asymptotic
$$\lim_{n\to\infty}(-1)^{mn}(z-z_{i_0})^{mn+1}\,\frac{P(D)^n\!\left(h(z)\right)}{(mn)!}
=\alpha_{i_0}\exp\!\left(-\frac{c_{m-1}}{m}(z-z_{i_0})\right),$$
and in particular
$$\lim_{n\to\infty}\left|\frac{P(D)^n\!\left(h(z)\right)}{(mn)!}\right|^{1/n}= \psi_S(z)^{-m}.$$
\end{lemma}

\begin{proof}
Write $q(x)^n=\sum_{k=0}^{mn}p_{n,k}x^k$, so that $P(D)^n=\sum_{k=0}^{mn}p_{n,k}D^k$.
Fix $z\in\mathbb C\setminus S$ and choose $0<\sigma<\psi_S(z)$.
Since $h$ is holomorphic on $\overline{\mathbb{D}(z,\sigma)}$, Cauchy's estimate gives
$$|h^{(k)}(z)|\le M_\sigma\,k!\,\sigma^{-k}\qquad (k\ge0),$$
where $M_\sigma:=\max_{|w-z|=\sigma}|h(w)|$.
Hence
\begin{equation}\label{eqPm_upper0}
\left|\frac{P(D)^n(h(z))}{(mn)!}\right|
\le M_\sigma\sum_{k=0}^{mn}|p_{n,k}|\frac{k!}{(mn)!}\,\sigma^{-k}.
\end{equation}

Reindex the sum by $k=mn-\ell$:
$$\sum_{k=0}^{mn}|p_{n,k}|\frac{k!}{(mn)!}\,\sigma^{-k}
=\sigma^{-mn}\sum_{\ell=0}^{mn}|p_{n,mn-\ell}|\frac{(mn-\ell)!}{(mn)!}\,\sigma^{\ell}.$$
By the coefficient bound in Lemma~\ref{lemDefect} (Appendix~\ref{app:coefficients}; here $[t^\ell]$ denotes the coefficient of $t^\ell$),
$$|p_{n,mn-\ell}|\frac{(mn-\ell)!}{(mn)!}\le [t^\ell]E(t),$$
where
$$E(t):=\exp\!\left(\frac{|c_{m-1}|}{m}t+\frac{|c_{m-2}|}{m}t^2+\cdots+\frac{|c_0|}{m}t^m\right).$$
Hence $\sum_{\ell\ge0}[t^\ell]E(t)\,\sigma^\ell=E(\sigma)$, so
$$\sum_{\ell=0}^{mn}|p_{n,mn-\ell}|\frac{(mn-\ell)!}{(mn)!}\,\sigma^{\ell}\le E(\sigma).$$
Combining with \eqref{eqPm_upper0} yields
$$\left|\frac{P(D)^n(h(z))}{(mn)!}\right|
\le M_\sigma\,\sigma^{-mn}E(\sigma).$$
Taking $n$th roots and letting $n\to\infty$ gives
$$\limsup_{n\to\infty}\left|\frac{P(D)^n(h(z))}{(mn)!}\right|^{1/n}\le \sigma^{-m}.$$
Finally, let $\sigma\uparrow\psi_S(z)$ to obtain the desired upper bound.

Now suppose that $z$ lies in the interior of the Voronoi cell of some $z_{i_0}\in S$, so that the nearest pole is unique: $\psi_S(z)=|z-z_{i_0}|<|z-z_j|$ for all $j\neq i_0$.
Write
$$h(w)=\frac{\alpha_{i_0}}{w-z_{i_0}}+g(w),$$
where $\alpha_{i_0}=\operatorname{Res}(h,z_{i_0})=A(z_{i_0})/B'(z_{i_0})\neq0$ and $g$ is holomorphic in a neighborhood of $\overline{\mathbb{D}(z,\psi_S(z)+\eta)}$ for some $\eta>0$ chosen so that $\overline{\mathbb{D}(z,\psi_S(z)+\eta)}$ contains no poles other than $z_{i_0}$; such an $\eta$ exists because uniqueness of the nearest pole implies $\min_{i\ne i_0}|z-z_i|>\psi_S(z)$.
Repeating the Cauchy estimate argument from the first part (now applied to $g$, with radius $\psi_S(z)+\eta$) gives
$$\limsup_{n\to\infty}\left|\frac{P(D)^n(g(z))}{(mn)!}\right|^{1/n}\le (\psi_S(z)+\eta)^{-m}<\psi_S(z)^{-m}.$$
It remains to analyze the principal part. Write $w-z_{i_0}=\zeta$ and define
$$u_n(\zeta):=\frac{1}{(mn)!}P(D)^n(\zeta^{-1}),$$
where we view $P(D)$ as acting on functions of $\zeta$ (since $d/dw=d/d\zeta$).
Since $D^k(\zeta^{-1})=(-1)^k k!\,\zeta^{-k-1}$ and $P(D)^n=\sum_{k=0}^{mn}p_{n,k}D^k$, we obtain
\begin{align*}
u_n(\zeta)
&=\frac{1}{(mn)!}\sum_{k=0}^{mn}p_{n,k}(-1)^k k!\,\zeta^{-k-1}\\
&=(-1)^{mn}\zeta^{-mn-1}\sum_{\ell=0}^{mn}(-1)^{\ell}p_{n,mn-\ell}\frac{(mn-\ell)!}{(mn)!}\,\zeta^{\ell}.
\end{align*}
Define
$$R_n(\zeta):=\sum_{\ell=0}^{mn}(-1)^{\ell}p_{n,mn-\ell}\frac{(mn-\ell)!}{(mn)!}\,\zeta^{\ell},$$
so that $u_n(\zeta)=(-1)^{mn}\zeta^{-mn-1}R_n(\zeta)$.

For each fixed $\ell$, Lemma~\ref{lemDefect} gives
$$a_{n,\ell}:=(-1)^{\ell}p_{n,mn-\ell}\frac{(mn-\ell)!}{(mn)!}\longrightarrow
\frac{1}{\ell!}\left(-\frac{c_{m-1}}{m}\right)^{\ell}.$$
Moreover, $|a_{n,\ell}|\le [t^{\ell}]E(t)$ (interpreting $a_{n,\ell}=0$ when $\ell>mn$), and
$\sum_{\ell\ge0}[t^{\ell}]E(t)\,R^{\ell}=E(R)<\infty$ for every $R>0$.
Therefore termwise convergence and uniform domination on $|\zeta|\le R$ imply the locally uniform limit
$$R_n(\zeta)=\sum_{\ell=0}^{mn}a_{n,\ell}\,\zeta^{\ell}\longrightarrow
\sum_{\ell\ge0}\frac{1}{\ell!}\left(-\frac{c_{m-1}}{m}\zeta\right)^{\ell}
=\exp\!\left(-\frac{c_{m-1}}{m}\zeta\right).$$
In particular, for our fixed $\zeta\neq0$ we have $R_n(\zeta)\to\exp\!\left(-\frac{c_{m-1}}{m}\zeta\right)$, and therefore
$$\lim_{n\to\infty}(-1)^{mn}\zeta^{mn+1}\,u_n(\zeta)
=\lim_{n\to\infty}R_n(\zeta)=\exp\!\left(-\frac{c_{m-1}}{m}\zeta\right).$$
Equivalently,
$$\lim_{n\to\infty}(-1)^{mn}\zeta^{mn+1}\,\frac{P(D)^n(\zeta^{-1})}{(mn)!}
=\exp\!\left(-\frac{c_{m-1}}{m}\zeta\right).$$

Since $P(D)^n(h(z))=\alpha_{i_0}P(D)^n(\zeta^{-1})+P(D)^n(g(z))$, it follows that
$$\lim_{n\to\infty}(-1)^{mn}\zeta^{mn+1}\,\frac{\alpha_{i_0}P(D)^n(\zeta^{-1})}{(mn)!}
=\alpha_{i_0}\exp\!\left(-\frac{c_{m-1}}{m}\zeta\right).$$
Moreover, the estimate for $g$ gives
$$\limsup_{n\to\infty}\left|\zeta^{mn+1}\frac{P(D)^n(g(z))}{(mn)!}\right|^{1/n}
\le \left(\frac{|\zeta|}{|\zeta|+\eta}\right)^m<1,$$
hence $\zeta^{mn+1}P(D)^n(g(z))/(mn)!\to0$.
Therefore
$$\lim_{n\to\infty}(-1)^{mn}\zeta^{mn+1}\,\frac{P(D)^n(h(z))}{(mn)!}
=\alpha_{i_0}\exp\!\left(-\frac{c_{m-1}}{m}\zeta\right).$$
Taking $n$th roots yields $\lim_{n\to\infty}\left|\frac{P(D)^n(h(z))}{(mn)!}\right|^{1/n}=|\zeta|^{-m}=\psi_S(z)^{-m}$.
\end{proof}

\smallskip
\begin{remark}
The pointwise asymptotic in Lemma~\ref{lem_growth_m} depends on $P(D)$ only through $c_{m-1}$: after dividing by $(mn)!$, the contributions involving $c_{m-2},\dots,c_0$ are subexponential in $n$ and hence do not affect the limit.
\end{remark}

\subsection{Main theorem ($c_0\neq0$)}

\begin{theorem}\label{thr5:mainm}
Let $h(z)=A(z)/B(z)$ be a reduced rational function, where $B$ is monic of degree $b\ge 2$ with distinct zeros $z_1,\dots,z_b$.
Set $S:=\{z_1,\dots,z_b\}$ and let $V:=V_S$ denote the Voronoi diagram of $S$ (Subsection~\ref{sub:sec2}).
Let
$$P(D)=D^m+c_{m-1}D^{m-1}+\cdots+c_0,\qquad m\ge1,\quad c_0\neq0,$$
and write
$$P(D)^n(h(z))=\frac{\widetilde A_n(z)}{B(z)^{mn+1}},$$
as in Lemma~\ref{lem_deg_Pm}.
Set
$$d_n:=\deg(\widetilde A_n)=a+bmn,\qquad a:=\deg A,$$
let $\mu_n:=\mu_{\widetilde A_n}$ be the zero-counting measure (Definition~\ref{def:zcmeasure}), and define
$$\widehat{\mathcal L}_{\mu_n}(z):=\mathcal L_{\mu_n}(z)-\frac{\log((mn)!)}{d_n}.$$
Then, as $n\to\infty$:
\begin{enumerate}
\item[(1)] $\widehat{\mathcal L}_{\mu_n}\to \Theta_{m,0}$ in $L^1_{\mathrm{loc}}(\mathbb C)$, where
$$\Theta_{m,0}(z)=\frac{b-1}{b}\,\Xi(z)-\frac{1}{mb}\log|c_0|,$$
and $\Xi$ is the B\o gvad--H\"agg potential from \eqref{defXi}.
\item[(2)] For every $z\in\mathbb C\setminus(V\cup S)$ we have $\mathcal L_{\mu_n}(z)\to+\infty$.
\item[(3)] $\mu_n$ converges vaguely on $\mathbb C$ to the canonical subprobability measure
$$\mu_{c,0}=\frac{b-1}{b}\,\mu_S,$$
supported on $V$; in particular $\mu_{c,0}(\mathbb C)=(b-1)/b$.
\end{enumerate}
\end{theorem}
The illustration of part $(3)$ of Theorem \ref{thr5:mainm} is shown in Figures \ref{f3}.
\begin{figure}[!hbt]
\centering
\includegraphics[width=0.65\linewidth]{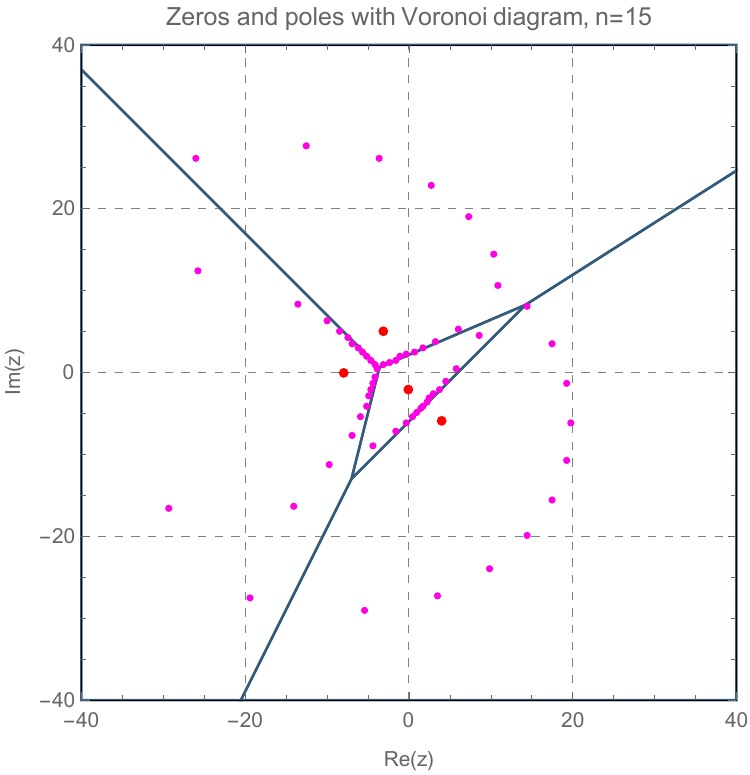}\hskip0.3cm
\caption{\footnotesize{The Voronoi diagram generated by the zeros of the polynomial
 $h(z)=(z+3-5i)(z-4+6i)(z+8)(z+2i)$ together with the zeros of $\left(P(D)\right)^{n}\left(1/h(z)\right)$ when
$ n=15$. Note that the zeros of $h$ are the poles of $1/h$ (red dots).}}\label{f3}
\end{figure}

\begin{proof}
We proceed in three steps.

\smallskip
\noindent\textbf{Step 1: pointwise convergence.}
Fix $z\in\mathbb C\setminus(V\cup S)$; recall $\psi_S(z)=\min_{1\le i\le b}|z-z_i|$.
Since $z\notin V$, the nearest pole is unique.
Thus Lemma~\ref{lem_growth_m} (second part) yields
$$\lim_{n\to\infty}\bigl|(mn)!^{-1}P(D)^n(h(z))\bigr|^{1/n}=\psi_S(z)^{-m}.$$
In particular, the limit is positive, so $P(D)^n(h(z))\neq0$ for all sufficiently large $n$.
By Lemma~\ref{lem_deg_Pm} we have
$$P(D)^n(h(z))=\frac{\widetilde A_n(z)}{B(z)^{mn+1}},$$
so for $z\notin S=\{z_1,\dots,z_b\}$,
$$\log|\widetilde A_n(z)|=\log|P(D)^n(h(z))|+(mn+1)\log|B(z)|.$$
By Remark~\ref{rem:logpot_poly},
$$\mathcal L_{\mu_n}(z)=\frac{1}{d_n}\log|\widetilde A_n(z)|-\frac{1}{d_n}\log\bigl|\mathrm{LC}(\widetilde A_n)\bigr|.$$
Lemma~\ref{lem_deg_Pm} gives $\mathrm{LC}(\widetilde A_n)=c_0^{\,n}\mathrm{LC}(A)$, and therefore
\begin{equation}\label{eqPm_shifted}
\widehat{\mathcal L}_{\mu_n}(z)
=\frac{1}{d_n}\log\bigl|(mn)!^{-1}P(D)^n(h(z))\bigr|
+\frac{mn+1}{d_n}\log|B(z)|
-\frac{n}{d_n}\log|c_0|
-\frac{1}{d_n}\log|\mathrm{LC}(A)|.
\end{equation}
Letting $n\to\infty$ and using $d_n\sim mnb$, $(mn+1)/d_n\to 1/b$, $n/d_n\to 1/(mb)$, and $\log|\mathrm{LC}(A)|/d_n\to0$, we obtain
$$\widehat{\mathcal L}_{\mu_n}(z)\to \Theta_{m,0}(z).$$

\smallskip
\noindent\textbf{Step 2: $L^1_{\mathrm{loc}}$ convergence.}
Each $\widehat{\mathcal L}_{\mu_n}$ is subharmonic on $\mathbb C$.

Fix a compact set $K\subset\mathbb C$ and choose $\varepsilon>0$ so small that the closed discs $\overline{\mathbb{D}(z_i,\varepsilon)}$ are pairwise disjoint.
Set
$$K_\varepsilon:=\Bigl(K\setminus\bigcup_{i=1}^b \mathbb{D}(z_i,\varepsilon)\Bigr)\cup\bigcup_{i=1}^b \partial\mathbb{D}(z_i,\varepsilon).$$
Then $K_\varepsilon$ is compact and avoids the pole set $S=\{z_1,\dots,z_b\}$, so with $\sigma:=\tfrac12\mathrm{dist}(K_\varepsilon,S)>0$ Cauchy's estimate gives a constant $M_K>0$ such that
$$\sup_{z\in K_\varepsilon}|h^{(k)}(z)|\le M_K\,k!\,\sigma^{-k}\qquad(k\ge0).$$
Combining this with Lemma~\ref{lemDefect} (exactly as in Lemma~\ref{lem_growth_m}) we obtain a constant $C_K$ (independent of $n$) such that
$$\sup_{z\in K_\varepsilon}\bigl|(mn)!^{-1}P(D)^n(h(z))\bigr|\le C_K\,\sigma^{-mn}.$$
Inserting this in \eqref{eqPm_shifted} (and using that $B$ is continuous and nonvanishing on $K_\varepsilon$) shows that
$$\sup_{z\in K_\varepsilon}\widehat{\mathcal L}_{\mu_n}(z)<\infty$$
uniformly in $n$.

Because each $\widehat{\mathcal L}_{\mu_n}$ is subharmonic, the maximum principle implies that its maximum on each closed disc $\overline{\mathbb{D}(z_i,\varepsilon)}$ is attained on the boundary circle $\partial\mathbb{D}(z_i,\varepsilon)\subset K_\varepsilon$.
Hence the same uniform upper bound holds on $\bigcup_{i=1}^b\overline{\mathbb{D}(z_i,\varepsilon)}$ and therefore on all of $K$.

Thus $\{\widehat{\mathcal L}_{\mu_n}\}$ is locally uniformly bounded above, so Hartogs' lemma \cite[Theorem~4.1.9]{Hormander} implies relative compactness in $L^1_{\mathrm{loc}}(\mathbb C)$.
Since Step~1 gives pointwise convergence to $\Theta_{m,0}$ off $V\cup S$ (a set of full planar Lebesgue measure), the only possible $L^1_{\mathrm{loc}}$ cluster point is $\Theta_{m,0}$.
Therefore $\widehat{\mathcal L}_{\mu_n}\to\Theta_{m,0}$ in $L^1_{\mathrm{loc}}(\mathbb C)$.

\smallskip
\noindent\textbf{Step 3: vague convergence of measures.}
Taking distributional Laplacians and using \eqref{tpm4} (note that adding constants does not change $\Delta$) gives
$$\mu_n=\frac{1}{2\pi}\Delta\widehat{\mathcal L}_{\mu_n}.$$
Hence for every $\varphi\in C_c^\infty(\mathbb C)$,
$$\int_{\mathbb C}\varphi\,d\mu_n
=\frac{1}{2\pi}\int_{\mathbb C}\widehat{\mathcal L}_{\mu_n}(z)\,\Delta\varphi(z)\,d\lambda(z)
\longrightarrow
\frac{1}{2\pi}\int_{\mathbb C}\Theta_{m,0}(z)\,\Delta\varphi(z)\,d\lambda(z)
=\int_{\mathbb C}\varphi\,d\left(\frac{1}{2\pi}\Delta\Theta_{m,0}\right),$$
where $d\lambda$ denotes planar Lebesgue measure and we used the $L^1_{\mathrm{loc}}$ convergence from Step~2.
Since $\sup_n\mu_n(\mathbb C)=1$ and $C_c^\infty(\mathbb C)$ is dense in $C_c(\mathbb C)$, this implies vague convergence to $\mu:=(2\pi)^{-1}\Delta\Theta_{m,0}$ (Definition~\ref{def:vague}).
By definition,
$$\Theta_{m,0}(z)=\frac{b-1}{b}\,\Xi(z)-\frac{1}{mb}\log|c_0|,$$
where $\Xi$ is the B\o gvad--H\"agg limit potential.
Since constants have zero Laplacian, Proposition~\ref{prop0} implies
$$\frac{1}{2\pi}\Delta\Theta_{m,0}=\frac{b-1}{b}\,\mu_S=\mu_{c,0},$$
which proves (3).

Finally, (2) follows because the shift term $\log((mn)!)/(mnb+a)\to+\infty$ while $\widehat{\mathcal L}_{\mu_n}(z)\to\Theta_{m,0}(z)$ is finite for $z\notin V\cup S$.
\end{proof}

\begin{remark}[Proof method]\label{rem:L1loc_method}
Our proof of $L^1_{\mathrm{loc}}$ convergence differs from \cite{B:H,Hagg}: rather than explicit integral estimates for logarithmic potentials, we use local uniform upper bounds to obtain relative compactness in $L^1_{\mathrm{loc}}$ (Step~2), and the pointwise limit on $\mathbb C\setminus(V\cup S)$ (Step~1) identifies the unique cluster point.
This approach extends directly to the case $r\ge1$ in Section~\ref{sec5}.
\end{remark}

\begin{corollary}\label{cor:away_from_V}
Let $K\Subset\mathbb C$ be compact with $K\cap V=\emptyset$.
Then $\mu_n(K)\to0$ as $n\to\infty$; equivalently, the number of zeros of $\widetilde A_n$ in $K$ is $o(d_n)$ as $n\to\infty$, where $d_n:=\deg(\widetilde A_n)$.
\end{corollary}

\begin{proof}
Since $V$ is closed and $K\cap V=\emptyset$, we have $\delta:=\mathrm{dist}(K,V)>0$.
Choose $\varphi\in C_c(\mathbb C)$ with $0\le \varphi\le1$, $\varphi\equiv1$ on $K$, and
$$\mathrm{supp}(\varphi)\subset\{z\in\mathbb C:\mathrm{dist}(z,K)<\delta/2\}\subset\mathbb C\setminus V.$$
Then $\int \varphi\,d\mu_{c,0}=0$ because $\mu_{c,0}$ is supported on $V$.
Hence
$$\mu_n(K)\le \int_{\mathbb C}\varphi\,d\mu_n \longrightarrow \int_{\mathbb C}\varphi\,d\mu_{c,0}=0,$$
which proves the claim.
\end{proof}

\begin{remark}\label{rem:specialcases}
The limiting measure $\mu_{c,0}=\frac{b-1}{b}\,\mu_S$ depends only on the pole set; the coefficients of $P(D)$ affect the limiting potential only through the additive constant $-(mb)^{-1}\log|c_0|$.
Since each $\mu_n$ is a probability measure, the fact that $\mu_{c,0}(\mathbb C)=(b-1)/b$ means that a mass $1/b$ escapes to infinity (in the sense of vague convergence).

For $m=1$, the operator identity $D+c_0=e^{-c_0 z}\circ D\circ e^{c_0 z}$ (where $e^{c_0 z}$ denotes multiplication by $e^{c_0 z}$) yields $(D+c_0)^n(h)=e^{-c_0 z}D^n(e^{c_0 z}h)$.
Thus Theorem~\ref{thr5:mainm} is equivalent to the $t=1$ case of H\"agg's theorem \cite{Hagg} applied to $f=h\,e^{c_0 z}$.
Normalization of $B$ and $P(D)$ is inessential; see the discussion in the Introduction.
\end{remark}

\section{The case $r\ge1$}\label{sec5}

In Section~\ref{sec4} we treated the case $r=0$.
Here we assume $r:=\operatorname{ord}_0 q\in\{1,\dots,m\}$, so $c_0=\cdots=c_{r-1}=0$ and $c_r\neq0$ (equivalently, $q(x)=x^r\widehat q(x)$ with $\widehat q(0)=c_r$).

After discarding the polynomial part of $h$ (Lemma~\ref{lem:kill_poly_part}), the argument of Theorem~\ref{thr5:mainm} carries over with two changes:
the degree normalization becomes $d_n=a+n(bm-r)$ (instead of $a+bmn$), and the correct shift of logarithmic potentials picks up an extra term $\log((rn)!)/d_n$ coming from the leading-coefficient factor $(a-b)_{rn}$ in Lemma~\ref{lem_deg_Pm}.
Thus the relevant shifted potentials are
$$\widehat{\mathcal L}_{\mu_n}(z):=\mathcal L_{\mu_n}(z)-\frac{\log((mn)!)-\log((rn)!)}{d_n},$$
which is identically $\,\mathcal L_{\mu_n}$ when $r=m$.

\subsection{Reducing to the proper case}

Write $h=A/B$ with $\gcd(A,B)=1$ and let $a=\deg A$, $b=\deg B$.
In contrast to the case $r=0$, when $r\ge1$ the iterates $P(D)^n$ eventually annihilate the polynomial part of $h$.
To avoid bookkeeping, we first reduce to the proper case $\deg A<b$.

\begin{lemma}\label{lem:kill_poly_part}
Assume $c_0=0$ and let $r\in\{1,\dots,m\}$ be minimal such that $c_r\neq0$.
Write $A=QB+R$ with $\deg R<b$ (Euclidean division). Then
$$h=\frac{A}{B}=Q+\frac{R}{B},$$
where $Q$ is the polynomial part of $h$ and $R/B$ is its proper part.
Then $P(D)^n(Q)=0$ for all $n$ with $rn>\deg Q$, and hence
$$P(D)^n(h)=P(D)^n(R/B)$$
for all such $n$.
In particular, for all sufficiently large $n$ we have $P(D)^n(h)=P(D)^n(R/B)$, so in our asymptotic problems we may replace $h$ by its proper part $R/B$ (which is still reduced and has the same simple poles).
\end{lemma}

\begin{proof}
Write $q(x)=x^r\widehat q(x)$ with $\widehat q(0)=c_r\neq0$.
Then $P(D)=q(D)=D^r\widehat q(D)$, and since $D$ commutes with $\widehat q(D)$ we have
$$P(D)^n(Q)=D^{rn}\bigl(\widehat q(D)^nQ\bigr).$$
The polynomial $\widehat q(D)^nQ$ has degree at most $\deg Q$, hence the right-hand side vanishes whenever $rn>\deg Q$.
The identity for $h=Q+R/B$ follows by linearity.
\end{proof}

Henceforth in this section we assume that $h=A/B$ is proper, i.e.\ $\deg A=a<b$; in particular $h(z)\to0$ as $z\to\infty$.

\subsection{The first nonzero coefficient and the modified degree}

Let $r:=\operatorname{ord}_0 q\in\{1,\dots,m\}$; equivalently,
\begin{equation}\label{eq:r_def}
q(x)=x^r\widehat q(x),\qquad \widehat q(0)=c_r\neq0.
\end{equation}
Then
$$P(D)=q(D)=D^{r}\widehat q(D),$$
and $D$ commutes with $\widehat q(D)$ since the coefficients are constant.

As a consequence of \eqref{eq:r_def}, writing
$$q(x)^n=\sum_{k=0}^{mn}p_{n,k}\,x^k$$
we have $p_{n,k}=0$ for $k<rn$ and $p_{n,rn}=c_r^{\,n}$.

We will use Lemma~\ref{lem:Ak_degree_LC} for the degree and leading coefficient of the derivative numerators $A_k$ (Lemma~\ref{lem:Ak}) in the proper case.

\subsection{Limit potentials and measures when $c_0=0$}

\begin{theorem}[The case $r\ge 1$]\label{thr:r_positive}
Let $h(z)=A(z)/B(z)$ be a reduced rational function, where $B$ is monic of degree $b\ge 2$ with distinct zeros $z_1,\dots,z_b$.
Set $S:=\{z_1,\dots,z_b\}$ and let $V:=V_S$ denote the Voronoi diagram of $S$ (Subsection~\ref{sub:sec2}).
Let
$$P(D)=\sum_{j=r}^{m}c_jD^j,\qquad m\ge1,\quad c_m=1,\quad c_r\neq 0,$$
and let $r:=\operatorname{ord}_0 q\in\{1,\dots,m\}$ (so that $c_0=\cdots=c_{r-1}=0$).
Assume that $h$ is proper, i.e., $\deg A=a<b$; by Lemma~\ref{lem:kill_poly_part} this entails no loss of generality for the asymptotic statements below.
Write
$$P(D)^n(h(z))=\frac{\widetilde A_n(z)}{B(z)^{mn+1}},$$
as in Lemma~\ref{lem_deg_Pm}, and set $d_n:=\deg(\widetilde A_n)=a+n(bm-r)$.
Let $\mu_n:=\mu_{\widetilde A_n}$ be the zero-counting measure (Definition~\ref{def:zcmeasure}), and define
$$\widehat{\mathcal L}_{\mu_n}(z):=\mathcal L_{\mu_n}(z)-\frac{\log((mn)!)-\log((rn)!)}{d_n},$$
noting that the correction term is identically $0$ when $r=m$.
Then, as $n\to\infty$:
\begin{enumerate}
\item[(1)] $\widehat{\mathcal L}_{\mu_n}\to \Theta_{m,r}$ in $L^1_{\mathrm{loc}}(\mathbb C)$, where
$$\Theta_{m,r}(z)=\frac{m(b-1)}{bm-r}\,\Xi(z)-\frac{1}{bm-r}\log|c_r|,$$
and $\Xi$ is the B\o gvad--H\"agg potential from \eqref{defXi}.
\item[(2)] If $r<m$, then for every $z\in\mathbb C\setminus(V\cup S)$ we have $\mathcal L_{\mu_n}(z)\to+\infty$.
\item[(3)] $\mu_n$ converges vaguely on $\mathbb C$ to the canonical subprobability measure
$$\mu_{c,r}=\frac{m(b-1)}{bm-r}\,\mu_S,$$
supported on $V$; in particular $\mu_{c,r}(\mathbb C)=\frac{m(b-1)}{bm-r}$.
When $r=m$ (i.e., $P(D)=D^m$), this reduces to $\mu_{c,m}=\mu_S$, recovering the B\o gvad--H\"agg theorem for the subsequence $h^{(mn)}$.
\end{enumerate}
\end{theorem}

The illustration of part $(3)$ of Theorem \ref{thr:r_positive} is shown in Figures \ref{f5}.
	
\begin{figure}[!hbt]
\centering
\includegraphics[width=0.65\linewidth]{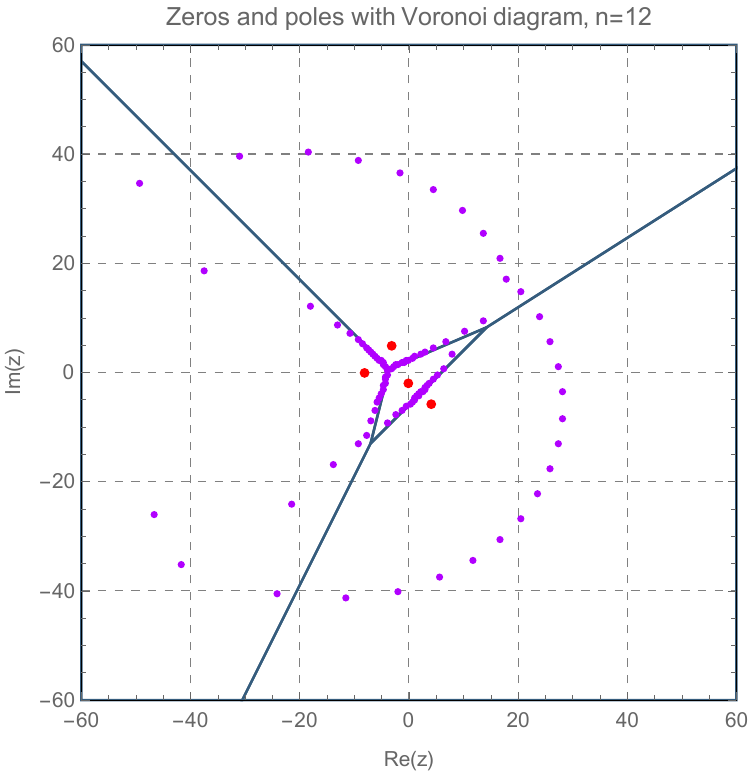}\hskip0.3cm
\caption{\footnotesize{The Voronoi diagram generated by the zeros of the polynomial 
$h(z)=(z+3-5i)(z-4+6i)(z+8)(z+2i)$ together with the zeros of $\left(P(D)\right)^{n}\left(1/h(z)\right)$
 when $n=12$.  Observe that the zeros of $h$ are the poles of $1/h$ (red dots).}}\label{f5}
\end{figure}

\begin{proof}
The proof is the same as Theorem~\ref{thr5:mainm}, except that we use the degree normalization $d_n=\deg(\widetilde A_n)=a+n(bm-r)$ and the leading coefficient $\mathrm{LC}(\widetilde A_n)=c_r^{\,n}\mathrm{LC}(A)(a-b)_{rn}$ from Lemma~\ref{lem_deg_Pm}.
We only indicate the changes.

\smallskip
\noindent\textbf{Step 1: pointwise convergence.}
Fix $z\in\mathbb C\setminus(V\cup S)$.
Lemma~\ref{lem_growth_m} yields
$$\lim_{n\to\infty}\bigl|(mn)!^{-1}P(D)^n(h(z))\bigr|^{1/n}=\psi_S(z)^{-m}.$$
Using $d_n=a+n(bm-r)$ and $\mathrm{LC}(\widetilde A_n)=c_r^{\,n}\mathrm{LC}(A)(a-b)_{rn}$ (Lemma~\ref{lem_deg_Pm}), we obtain
\begin{align}
\widehat{\mathcal L}_{\mu_n}(z)
&=\frac{1}{d_n}\log\bigl|(mn)!^{-1}P(D)^n(h(z))\bigr|
+\frac{mn+1}{d_n}\log|B(z)| \notag\\
&\qquad -\frac{n}{d_n}\log|c_r|
-\frac{1}{d_n}\log\bigl|\mathrm{LC}(A)(a-b)_{rn}\bigr|
+\frac{1}{d_n}\log((rn)!).\label{eqPm_shifted_r}
\end{align}
Since $a<b$, write $\gamma:=b-a>0$. Then
$$(a-b)_{rn}=(-\gamma)(-\gamma-1)\cdots(-\gamma-rn+1)=(-1)^{rn}\,\frac{(rn+\gamma-1)!}{(\gamma-1)!}.$$
Because $\gamma$ is fixed,
$$\log\bigl((rn+\gamma-1)!\bigr)-\log\bigl((rn)!\bigr)=\sum_{j=1}^{\gamma-1}\log(rn+j)=O(\log n),$$
and hence $\log|(a-b)_{rn}|=\log((rn)!)+O(\log n)$ as $n\to\infty$.
Therefore the last two terms in \eqref{eqPm_shifted_r} are $o(1)$, and letting $n\to\infty$ gives $\widehat{\mathcal L}_{\mu_n}(z)\to \Theta_{m,r}(z)$.

\smallskip
\noindent\textbf{Steps 2--3.}
The $L^1_{\mathrm{loc}}$ convergence and the vague convergence of measures follow exactly as in Theorem~\ref{thr5:mainm}.
In particular,
$$\frac{1}{2\pi}\Delta\Theta_{m,r}=\frac{m(b-1)}{bm-r}\,\mu_S=\mu_{c,r}.$$
Finally, when $r<m$ the shift term $(\log((mn)!)-\log((rn)!))/d_n\to+\infty$ off $V\cup S$, giving (2), while when $r=m$ the shift is identically $0$ and (1) yields $\mathcal L_{\mu_n}\to\Theta_{m,m}=\Xi$.
\end{proof}

\begin{corollary}\label{cor:away_from_V_r}
Let $K\Subset\mathbb C$ be compact with $K\cap V=\emptyset$.
Then $\mu_n(K)\to0$ as $n\to\infty$; equivalently, the number of zeros of $\widetilde A_n$ in $K$ is $o(d_n)$ with $d_n=a+n(bm-r)$.
\end{corollary}

\begin{proof}
The proof is identical to Corollary~\ref{cor:away_from_V}, using that the vague limit $\mu_{c,r}$ is supported on $V$.
\end{proof}

\begin{remark}\label{rem:c0}
The total mass satisfies $\mu_{c,r}(\mathbb C)=m(b-1)/(bm-r)\le1$, with equality if and only if $r=m$, i.e.\ $P(D)=D^m$.
In this extreme case $\Theta_{m,m}=\Xi$ and Theorem~\ref{thr:r_positive} reduces to the B\o gvad--H\"agg theorem (Theorem~\ref{thr1}) applied to the subsequence of derivatives of order $mn$.
More generally, the escaped mass equals
$$1-\mu_{c,r}(\mathbb C)=\frac{m-r}{bm-r},$$
so the amount of escape decreases as the order of vanishing of $q$ at $0$ increases.
Since $d_n=a+n(bm-r)$, this corresponds to about $(m-r)n$ zeros of $\widetilde A_n$ escaping to infinity.
\end{remark}

\section{Conclusion}\label{sec6}

We extend the measure-theoretic refinements of P\'olya's Shire theorem developed in \cite{B:H,Hagg} from pure derivatives to iterates of arbitrary monic constant-coefficient differential operators $P(D)$ of order $m\ge1$ acting on reduced rational functions with simple poles.
Let $q$ be the symbol of $P(D)$ and set $r:=\operatorname{ord}_0 q\in\{0,\dots,m\}$.
Theorems~\ref{thr5:mainm} and \ref{thr:r_positive} show that, after the appropriate degree normalization (and, when $r>0$, after discarding the polynomial part of $h$), the zero-counting measures of the numerator polynomials in $P(D)^n(h)$ converge vaguely to the canonical subprobability measure
$$\mu_{c,r}=\frac{m(b-1)}{bm-r}\,\mu_S \qquad (c\ \text{for ``canonical''}),$$
supported on the Voronoi diagram of the pole set.

When $r<m$, the unshifted logarithmic potentials diverge, but the factorially renormalized potentials
$$\mathcal L_{\mu_n}-\frac{\log((mn)!)-\log((rn)!)}{d_n}\qquad\bigl(d_n=\deg(\widetilde A_n)\bigr)$$
converge in $L^1_{\mathrm{loc}}(\mathbb C)$ to an explicit limit potential (and no renormalization is needed when $r=m$).
See also Remark~\ref{rem:specialcases} for the reduction to the case $m=1$, and Appendix~\ref{app:onepole} for the degenerate one-pole case $b=1$.

\medskip
\noindent\textbf{Further directions.}
Within the setting of rational functions with simple poles, the constant-coefficient case is now settled for all monic operators $P(D)$, including those with $c_0=0$.
Recent work of B\o gvad--Shapiro--Tahar--Warakkagun suggests that P\'olya-type ``Voronoi/Shire'' phenomena persist in much greater generality (e.g.\ on compact Riemann surfaces for iterations of a first-order operator $f\mapsto df/\omega$) \cite{BSTW}.
Natural next problems include: allowing multiple poles of $h$, treating higher-order poles and other singularities, and understanding to what extent similar Voronoi-skeleton limits persist for variable-coefficient differential operators or for certain classes of linear recurrence operators.

\subsection*{Acknowledgments}
The authors thank Boris Shapiro and Rikard B\o gvad for helpful discussions and corrections.

\appendix

\section{Auxiliary coefficient estimates for powers of polynomials}\label{app:coefficients}

\begin{lemma}[Defect coefficients]\label{lemDefect}
Let $q(x)=x^m+c_{m-1}x^{m-1}+\cdots+c_0$ and write $q(x)^n=\sum_{k=0}^{mn}p_{n,k}x^k$.
Fix $\ell\ge0$.
Then $p_{n,mn}=1$ and, for $\ell\ge1$, as $n\to\infty$,
$$p_{n,mn-\ell}=\binom{n}{\ell}\,c_{m-1}^{\,\ell}+O(n^{\ell-1}),$$
where the implicit constant depends on $q$ and $\ell$.
Moreover, for every $n\ge1$ and every $\ell\ge0$ with $\ell\le mn$,
$$\left|p_{n,mn-\ell}\right|\frac{(mn-\ell)!}{(mn)!}
\le [t^\ell]\exp\left(\frac{|c_{m-1}|}{m}t+\frac{|c_{m-2}|}{m}t^2+\cdots+\frac{|c_0|}{m}t^m\right),$$
where $[t^\ell]$ denotes the coefficient of $t^\ell$.
In particular, for every $t\ge0$,
$$\sum_{\ell=0}^{mn}\left|p_{n,mn-\ell}\right|\frac{(mn-\ell)!}{(mn)!}\,t^\ell
\le \exp\left(\frac{|c_{m-1}|}{m}t+\frac{|c_{m-2}|}{m}t^2+\cdots+\frac{|c_0|}{m}t^m\right).$$
\end{lemma}

\begin{proof}
The identity $p_{n,mn}=1$ is immediate.
Fix $\ell\ge1$.
To produce $x^{mn-\ell}$ in $q(x)^n$ we must choose, among the $n$ factors, lower-order terms whose total defect is $\ell$.
Write $s_j$ for the number of times we choose the term $c_{m-j}x^{m-j}$, so that
$$s_0+s_1+\cdots+s_m=n,
\qquad
s_1+2s_2+\cdots+ms_m=\ell,
\qquad
s:=s_1+\cdots+s_m\le \ell.$$
The multinomial expansion gives
$$p_{n,mn-\ell}
=\sum_{\substack{s_1,\dots,s_m\ge0\\ s_1+2s_2+\cdots+ms_m=\ell}}
\frac{n!}{(n-s)!\,s_1!\cdots s_m!}\,c_{m-1}^{\,s_1}\cdots c_0^{\,s_m}.$$
The tuple $(s_1,\dots,s_m)=(\ell,0,\dots,0)$ yields the main term $\binom{n}{\ell}c_{m-1}^{\,\ell}$, and every other admissible tuple satisfies $s\le \ell-1$, so its multinomial prefactor is $O(n^{\ell-1})$.
This proves the asymptotic formula.

For the coefficient bound, multiply each summand by $(mn-\ell)!/(mn)!$ and observe that
$$\frac{n!}{(n-s)!}\,\frac{(mn-\ell)!}{(mn)!}
\le \frac{n!}{(n-s)!}\,\frac{(mn-s)!}{(mn)!}
=\prod_{u=0}^{s-1}\frac{n-u}{mn-u}\le m^{-s},$$
since $mn-u\ge m(n-u)$ for $0\le u\le s-1$.
Therefore
$$\left|p_{n,mn-\ell}\right|\frac{(mn-\ell)!}{(mn)!}
\le \sum_{\substack{s_1,\dots,s_m\ge0\\ s_1+2s_2+\cdots+ms_m=\ell}}
\prod_{j=1}^{m}\frac{1}{s_j!}\left(\frac{|c_{m-j}|}{m}\right)^{s_j}.$$
The right-hand side is exactly the coefficient of $t^\ell$ in
$\prod_{j=1}^{m}\exp\!\bigl((|c_{m-j}|/m)t^j\bigr)$, which equals the claimed exponential.
The final generating-function inequality follows by multiplying the coefficient bound by $t^\ell$ and summing over $\ell$.
\end{proof}

\section{The one-pole case}\label{app:onepole}

Although we assume $b\ge2$ in the main theorems, the case $b=1$ is degenerate.
If $r=m$ (so $P(D)=D^m$), then $P(D)^n(h)=h^{(mn)}$ has constant numerator and the normalized zero-counting measures are not informative.
In the remaining case $r<m$, the numerator degrees tend to $\infty$ and all mass escapes to infinity.

\begin{proposition}[One pole]\label{prop:onepole}
Let $h=A/B$ be reduced with $B(z)=z-z_1$ (so $h$ has exactly one simple pole), and let $P(D)$ be a monic constant-coefficient differential operator of order $m\ge1$ with symbol $q$.
Let $r=\operatorname{ord}_0 q<m$ and let $c_r$ be the first nonzero coefficient of $q$.
If $r>0$, we replace $h$ by its proper part (Lemma~\ref{lem:kill_poly_part}), so that $\deg A=0$.
For $n\ge1$ write
$$P(D)^n(h(z))=\frac{\widetilde A_n(z)}{(z-z_1)^{mn+1}},$$
set $d_n:=\deg(\widetilde A_n)$, and let $\mu_n:=\mu_{\widetilde A_n}$.
Define the shifted potentials
$$\widehat{\mathcal L}_{\mu_n}(z):=\mathcal L_{\mu_n}(z)-\frac{\log((mn)!)-\log((rn)!)}{d_n}.$$
Then $d_n=a+n(m-r)\to\infty$, where $a=\deg A$ (so $a=0$ when $r>0$ after passing to the proper part), and:
\begin{enumerate}
\item[(i)] $\widehat{\mathcal L}_{\mu_n}\to -\frac{1}{m-r}\log|c_r|$ in $L^1_{\mathrm{loc}}(\mathbb C)$;
\item[(ii)] $\mu_n\to0$ vaguely on $\mathbb C$ (all mass escapes to infinity).
\end{enumerate}
\end{proposition}

\begin{proof}
If $r>0$, write $A=Q(z-z_1)+R$ with $R\neq0$ constant.
By Lemma~\ref{lem:kill_poly_part} (applied with $B(z)=z-z_1$) we have $P(D)^n(h)=P(D)^n(R/(z-z_1))$ for all sufficiently large $n$.
Hence we may replace $h$ by its proper part and assume $\deg A=a=0$.

Fix $z\neq z_1$.
Since the nearest pole is unique, Lemma~\ref{lem_growth_m} yields
$$\lim_{n\to\infty}\bigl|(mn)!^{-1}P(D)^n(h(z))\bigr|^{1/n}=|z-z_1|^{-m}.$$
Because $d_n=a+n(m-r)$, this implies
$$\frac{1}{d_n}\log\bigl|(mn)!^{-1}P(D)^n(h(z))\bigr|\longrightarrow -\frac{m}{m-r}\log|z-z_1|.$$
On the other hand, Lemma~\ref{lem_deg_Pm} (with $b=1$) gives
$$\mathrm{LC}(\widetilde A_n)=c_r^n\,\mathrm{LC}(A)\,(a-1)_{rn}.$$
Arguing as in \eqref{eqPm_shifted_r} we obtain
\begin{align*}
\widehat{\mathcal L}_{\mu_n}(z)
&= \frac{1}{d_n}\log\bigl|(mn)!^{-1}P(D)^n(h(z))\bigr|+\frac{mn+1}{d_n}\log|z-z_1|\\
&\qquad -\frac{n}{d_n}\log|c_r|-\frac{\log|\mathrm{LC}(A)|}{d_n}-\frac{\log|(a-1)_{rn}|-\log((rn)!)}{d_n}.
\end{align*}
If $r>0$ then $a=0$ and $(a-1)_{rn}=(-1)^{rn}(rn)!$, so the last term is identically $0$; if $r=0$ it is also $0$.
The two $\log|z-z_1|$ terms cancel in the limit, and therefore
$$\widehat{\mathcal L}_{\mu_n}(z)\longrightarrow -\frac{1}{m-r}\log|c_r|$$
for each $z\neq z_1$.

As in Step~2 of Theorem~\ref{thr5:mainm}, Hartogs' lemma implies that the family $\{\widehat{\mathcal L}_{\mu_n}\}$ is relatively compact in $L^1_{\mathrm{loc}}(\mathbb C)$.
Since the pointwise limit holds on $\mathbb C\setminus\{z_1\}$, we get $\widehat{\mathcal L}_{\mu_n}\to -\frac{1}{m-r}\log|c_r|$ in $L^1_{\mathrm{loc}}(\mathbb C)$.
Taking distributional Laplacians then gives $\mu_n\to0$ vaguely.
\end{proof}

\begin{remark}
If $h$ has a single pole at $z_1$ of order $p>1$, write its principal part as a linear combination of derivatives of $(z-z_1)^{-1}$ and use that $P(D)$ commutes with $D$.
The same argument shows that, whenever the numerator degrees tend to $\infty$, the normalized zero-counting measures have no nontrivial finite vague limit on $\mathbb C$; all mass escapes to infinity (compare \cite{PJS1}).
\end{remark}

\end{document}